\documentclass[leqno]{article}
\usepackage[english,frenchb]{babel}
\usepackage{lmodern}
\usepackage{enumerate}
\usepackage{amsmath}
\usepackage{amsfonts}
\usepackage{amssymb}
\usepackage{amsthm}
\usepackage{graphicx}
\usepackage{mathrsfs}
\usepackage{pxfonts}
\usepackage{datetime}
\usepackage{comment}
\usepackage{fourier-orns}
\usepackage{dsfont}
\usepackage{xcolor}
\definecolor{darkpastelblue}{rgb}{0.47, 0.62, 0.8}
\usepackage[cyr]{aeguill}
\usepackage{fancyhdr}
\usepackage{mathabx}
\usepackage[latin1]{inputenc}
\usepackage[all,cmtip]{xy}

\usepackage{tikz-cd}
\usetikzlibrary{decorations.pathmorphing}

\usepackage{amsfonts,amsmath,graphicx}

\usepackage{geometry}
\newcommand{\bc}{\begin{center}}
\newcommand{\ec}{\end{center}}
\newcommand{\bq}{\begin{quote}}
\newcommand{\eq}{\end{quote}}
\newcommand{\bqtn}{\begin{quotation}}
\newcommand{\eqtn}{\end{quotation}}
\newcommand{\beq}{\begin{equation}}
\newcommand{\eeq}{\end{equation}}
\newcommand{\bearr}{\begin{eqnarray}}
\newcommand{\eearr}{\end{eqnarray}}
\newcommand{\bearrn}{\begin{eqnarray*}}
\newcommand{\eearrn}{\end{eqnarray*}}
\newcommand{\bi}{\begin{itemize}}
\newcommand{\ei}{\end{itemize}}
\newcommand{\be}{\begin{enumerate}}
\newcommand{\ee}{\end{enumerate}}
\newcommand{\bthe}{\begin{theorem}}
\newcommand{\ethe}{\end{theorem}}
\newcommand{\blem}{\begin{lemme}}
\newcommand{\elem}{\end{lemme}}
\newcommand{\bsolu}{\begin{solution}}
\newcommand{\esolu}{\end{solution}}
\newcommand{\bexer}{\begin{exercise}}
\newcommand{\eexer}{\end{exercise}}

\newcommand{\ba}{\begin{array}}
\newcommand{\ea}{\end{array}}

\usepackage{amsfonts,amsmath}
\newtheorem{theoreme}{Theorem}[section]
\newtheorem{theorem}[theoreme]{Theorem}
\newtheorem{lemme}[theoreme]{Lemma}
\newtheorem{lemma}[theoreme]{Lemma}
\newtheorem{proposition}[theoreme]{Proposition}
\newtheorem{definition}[theoreme]{Definition}

\newtheorem{corollaire}[theoreme]{Corollary}

\newtheorem{solution}[theoreme]{Solution}
\newtheorem{exercise}[theoreme]{Exercise}
\newcommand{\bdefi}{\begin{definition}}
\newcommand{\edefi}{\end{definition}}
\newcommand{\brk}{\begin{remarque}}
\newcommand{\erk}{\end{remarque}}
\newcommand{\bpp}{\begin{proposition}}
\newcommand{\epp}{\end{proposition}}
\newcommand{\bpf}{\begin{proof}}
\newcommand{\epf}{\end{proof}}
\newcommand{\bcor}{\begin{corollaire}}
\newcommand{\ecor}{\end{corollaire}}
\newcommand{\bsol}{\begin{solution}}
\newcommand{\esol}{\end{solution}}
\theoremstyle{definition}

\newtheorem{remarque}[theoreme]{Remark}
\allowdisplaybreaks
\title{An explicit triangular integral basis for any separable cubic extension of a function field}
\author{Sophie Marques and Kenneth Ward}

\begin{document}
\large
\selectlanguage{english}
\maketitle
\begin{abstract}
We determine an explicit triangular integral basis for any separable cubic extension of a rational function field over a finite field in any characteristic. We obtain a formula for the discriminant of every such extension in terms of a standard form in a tower for the Galois closure.
\end{abstract} 
Let $ \mathbb{F}_q(x)$ be a rational field in one variable, where $ q= p^n$ and $p$ is a prime integer. In this paper, our main theorems determine explicitly a triangular integral basis in any characteristic for any separable extension: Theorem \ref{3primetoqbasis} for characteristic different from $3$ and Theorem \ref{3dividesqbasis} for characteristic $3$. The algorithm proposed only depends only on the factorisation of polynomials, Hasse's algorithm to obtain a global standard form in the sense of Artin-Schreier theory (see \cite[Example 5.8.9]{Vil} or \cite{Has}) and the Chinese remainder theorem. In any characteristic, obtaining generators in a type of standard form in a tower generated by the Galois closure of a cubic appears to be key to explicitly determining this triangular integral basis. We note that this work generalises the results begun in \cite[Theorem 5.17]{MWcubic}, and that our work in characteristic $p \geq 5$ is a consequence of our computations and results appearing in \cite{lotsofpeople, Scheidler}. In characteristic 3, the basis we construct bears some resemblance to that of \cite[Theorem 9]{MadMad}, although the construction is different.

In \S 1, we compute explicitly the discriminant for any separable cubic function field over $\mathbb{F}_q(x)$. In \S 2, we use those computations to determine that a triangular integral basis always exists in terms of our choices of generators, and we determine such a basis explicitly in each characteristic.

\section{Discriminant of a cubic extension} 
Let $\mathfrak{D}_{L/\mathbb{F}_q(x)}$ be the different and $\partial_{L/\mathbb{F}_q(x)}$ the discriminant divisor of $L/\mathbb{F}_q(x)$ \cite[Section 5.6]{Vil}. We have by\cite[Definition 5.6.8]{Vil} that
\begin{equation} \label{normdisc} (\partial_{L/\mathbb{F}_q(x)})_{\mathbb{F}_q[x]}=N_{L/\mathbb{F}_q(x)}((\mathfrak{D}_{L/\mathbb{F}_q(x)})_{\mathbb{F}_q[x]}) = \prod_{\mathfrak{p} \in \mathcal{R}} \mathfrak{p}^{\sum_{\mathfrak{P}|\mathfrak{p}} d(\mathfrak{P}|\mathfrak{p})f(\mathfrak{P}|\mathfrak{p})},\end{equation}
where $\mathcal{R}$ is the set of finite places of $\mathbb{F}_q(x)$ ramified in $L$, $d(\mathfrak{P}|\mathfrak{p})$ is the differential exponent of $\mathfrak{P}|\mathfrak{p}$, and $f(\mathfrak{P}|\mathfrak{p})$ is the inertia degree. Let $\mathcal{O}_{L,x}$ denote the integral closure of $\mathbb{F}_q[x]$ in $L$. If $\omega \in \mathcal{O}_{L,x}$ with minimal polynomial $$X^3 + b X^2 + cX + d = 0,$$ then the discriminant $\Delta (\omega)$ of $\omega$ is defined \cite[Theorem 2.1]{Con} as $$\Delta (\omega) = b^2 c^2 - 4 c^3 - 4b^3 d - 27 d^2 + 18bcd.$$ The ideal $\text{Ind}(\omega)$ of $\mathbb{F}_q[x]$ is defined \cite[p. 789]{lotsofpeople} by $$( \Delta(\omega))_{\mathbb{F}_q[x]} = \text{Ind}(\omega)^2 (\partial_{L/\mathbb{F}_q(x)})_{\mathbb{F}_q[x]}.$$ We suppose that $L/\mathbb{F}_q(x)$ is an impure cubic extension. We split the calculations into the cases $3 \nmid q$ and $3 \mid q$, as the associated generating equations are different \cite[Corollary 1.3]{MWcubic2}.

\subsection{$p \neq 3$} Let $y$ be a generator of $L/\mathbb{F}_q(x)$ which satisfies the equation $y^3 - 3y - a=0$, where $a \in \mathbb{F}_q(x)$, which exists by \cite[Corollary 1.3]{MWcubic2}. We let $\mathfrak{p}$ a place of $K$ and $\mathfrak{P}$ a place of $L$ above $\mathfrak{p}$.  We let $a= \alpha/(\gamma^3 \beta)$, where $(\alpha,\beta\gamma) = 1$, $\beta$ is cube-free, and $\beta = \beta_1 \beta_2^2$, where $\beta_1$ and $\beta_2$ are square-free. Let $\omega = \gamma \beta_1 \beta_2 y$. The minimal polynomial of $\omega$ over $\mathbb{F}_q(x)$ is equal to
$$X^3 - 3\gamma^2 \beta_1^2 \beta_2^2 X- \alpha \beta_1^2 \beta_2,$$ 
and $\omega$ is integral over $\mathbb{F}_q[x]$. By definition, the discriminant of $\omega$ is equal to $$\Delta(\omega) = 108 \gamma^6 \beta_1^6 \beta_2^6 - 27 \alpha^2 \beta_1^4 \beta_2^2 = 27 \beta_1^4 \beta_2^2 (4\gamma^6 \beta_1^2 \beta_2^4 - \alpha^2)= 27 \beta_1^4 \beta_2^2 (4\gamma^6 \beta^2 - \alpha^2).$$  We have (up to constant multiples) \begin{equation} \label{ind} (\beta_1^4 \beta_2^2 (4\gamma^6 \beta^2 - \alpha^2))_{\mathbb{F}_q[x]} =( \Delta(\omega))_{\mathbb{F}_q[x]} = \text{Ind}(\omega)^2 (\partial_{L/\mathbb{F}_q(x)})_{\mathbb{F}_q[x]}.\end{equation} In the following lemma, we determine $(\partial_{L/\mathbb{F}_q(x)})_{\mathbb{F}_q[x]}$ and $\text{Ind}(\omega)$ when $p\neq 3$, for which the cases $p \neq 2$ and $p = 2$ are different.

\begin{lemme} \label{indlemma}
\begin{enumerate}[(a)]
	\item If $p \neq 2$, let $(4\gamma^6 \beta^2 - \alpha^2)=\eta_1 \eta_2^2$, where $\eta_1$ is square-free. Then $$(\partial_{L/\mathbb{F}_q(x)})_{\mathbb{F}_q[x]} = ((\beta_1 \beta_2)^2 \eta_1)_{\mathbb{F}_q[x]} \quad \text{ and }\quad \text{\emph{Ind}}(w) = (\beta_1 \eta _2)_{\mathbb{F}_q[x]}.$$
	\item If $p = 2$, let $r$ be a root of the quadratic resolvent of the polynomial $X^3 -3X-a$. One can find an Artin-Schreier generator $z$ in global standard form for $\mathbb{F}_q(x,r)/ \mathbb{F}_q(x)$ \cite[Example 5.8.8]{Vil}, i.e., an Artin-Schreier generator such that $$z^2 - z =  b,$$ where $b=\frac{1}{a^2} + 1 + \eta^2 - \eta$ for some $\eta \in \mathbb{F}_q(x)$ such that each finite place $\mathfrak{p}$ of $\mathbb{F}_q(x)$ which ramifies in $\mathbb{F}_q(x,z)=\mathbb{F}_q(x,r)$ satisfies 
  $$v_\mathfrak{p}\left(b\right) < 0, \quad \quad \left(v_\mathfrak{p}\left(b \right) ,2\right) = 1,$$ 
  and any unramified place $\mathfrak{p}\mid  ( \alpha )_{\mathbb{F}_q[x]}$ satisfies $$v_\mathfrak{p}\left( b \right) \geq 0.$$ We let $\ell_\mathfrak{p} = -v_\mathfrak{p}\left(b \right)$ for each ramified $\mathfrak{p} \mid (\alpha )_{\mathbb{F}_q[x]}$, and otherwise we let $\ell_\mathfrak{p} = -1$ whenever $\mathfrak{p} \mid  (\alpha )_{\mathbb{F}_q[x]}$ is unramified. Then  $$(\partial_{L/\mathbb{F}_q(x)})_{\mathbb{F}_q[x]} = ((\beta_1 \beta_2)^2 )_{\mathbb{F}_q[x]}  \prod_{\substack{ (v_\mathfrak{p}(b),2) = 1 \\ and \ v_\mathfrak{p}(b)<0 }}  \mathfrak{p}^{\ell_\mathfrak{p} + 1} \quad \text{ and } \quad \text{\emph{Ind}}(\omega) = (\beta_1 )_{\mathbb{F}_q[x]} \prod_{\mathfrak{p} \mid (\alpha )_{\mathbb{F}_q[x]}} \mathfrak{p}^{v_\mathfrak{p}(\alpha) -\frac{1}{2} (\ell_\mathfrak{p} + 1)}.$$

\end{enumerate}
	
\end{lemme}
\begin{proof}
        \begin{enumerate} \item When $p \neq 2$ we have $d(\mathfrak{P}|\mathfrak{p}) = 2$ whenever  $e(\mathfrak{P}|\mathfrak{p}) = 3$, which occurs when $v_\mathfrak{p}(a) < 0$ and $(v_\mathfrak{p}(a),3)=1$. We also have $d(\mathfrak{P}|\mathfrak{p}) = 1$ if, and only if, $v_\mathfrak{p}(4\gamma^6 \beta^2- \alpha^2)$ is odd (see \cite[Corollary 3.15]{MWcubic2}). In either of these cases, only one place $\mathfrak{P}|\mathfrak{p}$ above has $d(\mathfrak{P}| \mathfrak{p})
       \neq 0$, and for this place $\mathfrak{P}$, $f(\mathfrak{P}|\mathfrak{p}) = 1$. Thus we obtain from \eqref{normdisc} that $$(\partial_{L/\mathbb{F}_q(x)})_{\mathbb{F}_q[x]} = \prod_{\substack{ \mathfrak{p} |( \beta )_{\mathbb{F}_q[x]}}} \mathfrak{p}^2  \prod_{\substack{ \mathfrak{p} |(\alpha^2 - 4\gamma^3 \beta^2 )_{\mathbb{F}_q[x]}\\ (v_\mathfrak{p}(\alpha^2 - 4\gamma^3 \beta^2),2) = 1}} \mathfrak{p} = ((\beta_1 \beta_2)^2 \eta_1)_{\mathbb{F}_q[x]}.$$ Thus we obtain from \eqref{ind} that \begin{align*} \text{Ind}(\omega)^2 &= (\beta_1^4 \beta_2^2 \eta_1 \eta_2^2)_{\mathbb{F}_q[x]} (\partial_{L/\mathbb{F}_q(x)})_{\mathbb{F}_q[x]} ^{-1} \\&=( \beta_1^4 \beta_2^2 \eta_1 \eta_2^2   ( \beta_1 \beta_2)^{-2} \eta_1^{-1} )_{\mathbb{F}_q[x]} \\& = (\beta_1 \eta_2)^2_{\mathbb{F}_q[x]} . \end{align*}
Therefore, $\text{Ind}(\omega) = (\beta_1 \eta_2)_{\mathbb{F}_q[x]} $.
  \item 	
  When $p=2$, then again $d(\mathfrak{P}|\mathfrak{p}) = 2$ whenever  $e(\mathfrak{P}|\mathfrak{p}) = 3$, which occurs when $v_\mathfrak{p}(a) < 0$ and $(v_\mathfrak{p}(a),3)=1$.  Let $b$ be as in the statement of the lemma. By Artin-Schreier theory \cite[Example 5.8.8]{Vil}, only finite places $\mathfrak{p}\mid  ( \alpha )_{\mathbb{F}_q[x]}$ and possibly the place at infinity can be ramified in $\mathbb{F}_q(x,r)/\mathbb{F}_q(x)$. More precisely, by \cite[Lemma 3.23]{MWcubic2}, we know that $ e( \mathfrak{P}| \mathfrak{p})=2$ if and only if  $v_\mathfrak{p}\left(b\right) < 0, \ \left(v_\mathfrak{p}\left(b \right) ,2\right) = 1$.
  Then for such $\mathfrak{p}$, we have (Ibid.) that $$d(\mathfrak{P}|\mathfrak{p}) = -v_\mathfrak{p}\left(b\right)+1 = \ell_\mathfrak{p}+1,$$ and similarly for $\mathfrak{p} \mid (\alpha)_{\mathbb{F}_q[x]}$ unramified by definition of  $\ell_\mathfrak{p}$. Once again, in either ramified case, there is only one place $\mathfrak{P}|\mathfrak{p}$ with $d(\mathfrak{P}| \mathfrak{p})\neq 0$, and for this place, $f(\mathfrak{P}|\mathfrak{p})=1$ for each such place. 
  Thus, we obtain from \eqref{normdisc} that
   $$(\partial_{L/\mathbb{F}_q(x)})_{\mathbb{F}_q[x]} = \prod_{\substack{ \mathfrak{p} | (\beta)_{\mathbb{F}_q[x]} }} \mathfrak{p}^2  \prod_{\substack{\mathfrak{p} | ( b)_{\mathbb{F}_q[x]}\\ (v_\mathfrak{p}(b),2) = 1 \\ and \ v_\mathfrak{p}(b)<0 }} \mathfrak{p}^{\ell_\mathfrak{p} + 1} = ((\beta_1 \beta_2)^2 )_{\mathbb{F}_q[x]}  \prod_{\substack{ \mathfrak{p} | ( b)_{\mathbb{F}_q[x]}\\(v_\mathfrak{p}(b),2) = 1 \\ and \ v_\mathfrak{p}(b)<0 }}  \mathfrak{p}^{\ell_\mathfrak{p} + 1}.$$
  From \eqref{ind}, we then find $$(\beta_1^4 \beta_2^2\alpha^2 )_{\mathbb{F}_q[x]}= (\Delta(\omega) )_{\mathbb{F}_q[x]}= \text{Ind}(\omega)^2 (\partial_{L/\mathbb{F}_q(x)})_{\mathbb{F}_q[x]} = \text{Ind}(\omega)^2 (\beta_1 \beta_2)_{\mathbb{F}_q[x]}^2 \prod_{\mathfrak{p} \mid (\alpha )_{\mathbb{F}_q[x]}} \mathfrak{p}^{\ell_\mathfrak{p}+1},$$ and hence (note $\frac{1}{2}(\ell_\mathfrak{p} + 1) \in \mathbb{Z}$ since $(\ell_{\mathfrak{p}}, 2)=1$) that $$\text{Ind}(\omega) = (\beta_1 )_{\mathbb{F}_q[x]} \prod_{\mathfrak{p} \mid (\alpha )_{\mathbb{F}_q[x]}} \mathfrak{p}^{v_\mathfrak{p}(\alpha) - \frac{1}{2}(\ell_\mathfrak{p} + 1)}.$$
  
 \end{enumerate}
 \end{proof}
\subsection{$p=3$}

We now treat the case where $p=3$. Let $y_1$ be a generator of $L/\mathbb{F}_q(x)$ such that $y_1^3 +a_1 y_1 + a_1^2=0$, which exists by \cite[Corollary 1.3]{MWcubic2}. We first establish the existence of a generator in standard form at any finite place.  This will allow us to use \cite[Theorem 3.19/20]{MWcubic2} to determine $(\partial_{L/\mathbb{F}_q(x)})_{\mathbb{F}_q[x]}$, which will be useful for finding an integral basis of  $\mathcal{O}_{L,x}$. 

\begin{lemme} \label{3standard} There exists $b \in \mathbb{F}_q(x)$ and generator $z$ of $L/\mathbb{F}_q(x)$ with minimal polynomial $$X^3 + b X + b^2 = 0$$ for which, for all finite places of  $\mathbb{F}_q(x)$,$$ v_\mathfrak{q}(b) \geq 0, \quad \text{ or } \quad v_\mathfrak{q}(b) < 0 \text{ and } (v_\mathfrak{q}(b) ,3) = 1.$$ \end{lemme} 

\begin{proof} We prove the lemma by way of the following algorithm. Let $\mathfrak{p}$ be a finite place of $\mathbb{F}_q(x)$ for which $v_\mathfrak{p}(a_1) < 0$ and $3 \mid v_\mathfrak{p}(a_1)$. We will find an element $a \in \mathbb{F}_q(x)$ and a generator $y$ of $L/\mathbb{F}_q(x)$ with minimal polynomial $$X^3 + a X + a^2 = 0,$$ and such that \\
    \begin{itemize} \item $v_\mathfrak{p}(a) \geq 0$ or $v_\mathfrak{p}(a) < 0$ and $(v_\mathfrak{p}(a),3) = 1$; \item all finite places $\mathfrak{q} \neq \mathfrak{p}$ such that $v_\mathfrak{q}(a_1) < 0$ also have $v_\mathfrak{q}(a) = v_\mathfrak{q}(a_1)$; \item all finite places $\mathfrak{q}$ such that $v_\mathfrak{q}(a_1) \geq 0$ also have $v_\mathfrak{q}(a) \geq 0$. \end{itemize} 
    \vspace{4mm}
  By \cite[Lemma 1.11]{MWcubic2}, any other generator $y_2$ of $L/\mathbb{F}_q(x)$ with a minimal equation of the same form $y_2^3 +a_2 y_2 +a_2^2=0$ is such that $y_2 =-\beta (\frac{j}{a_1}y_1+ \frac{1}{a_1} w_2)$ for some $w_2 \in K$, and we have $$a_2  = \frac{(ja_1^2 + (w_2^3 + a_1 w_2) )^2}{a_1^3}.$$ Let $\lambda \in \mathbb{F}_q(x)$ such that $v_\mathfrak{p}(\lambda) = 2v_\mathfrak{p}(a_1)/3$. By strong approximation \cite[Chap. 2.15]{CasFro}, there exists $\alpha \in \mathbb{F}_q(x)$ such that \\

\begin{itemize} \item $v_\mathfrak{p}(\alpha-\lambda) \geq  \left(2v_\mathfrak{p}(a_1)/3\right) + 1$ and \item $v_\mathfrak{q}(\alpha) \geq 0$ at all other finite places $\mathfrak{q} \neq \mathfrak{p}$ of $\mathbb{F}_q(x)$. \end{itemize} 
\vspace{4mm}

As $v_\mathfrak{p}(\alpha- \lambda) \geq  \left(2v_\mathfrak{p}(a_1)/3\right) + 1 = v_\mathfrak{p}(\lambda) + 1 > v_\mathfrak{p}(\lambda)$, the first condition implies by the non-Archimedean property $$v_\mathfrak{p}(\alpha) =v_\mathfrak{p}(\alpha- \lambda + \lambda) = \min\{v_\mathfrak{p}(\alpha- \lambda),v_\mathfrak{p}(\lambda) \} = v_\mathfrak{p}(\lambda) = 2v_\mathfrak{p}(a_1)/3.$$ We may therefore apply \cite[Lemma 3.18]{MWcubic2}, where $w_0 \in K$ was chosen so that $w_0 \neq - \alpha^{-3} ja_1^2$ and $$v_\mathfrak{p}(\alpha^{-3} ja_1^2 + w_0) > 0.$$ 

As $p = 3$, the map $X \rightarrow X^3$ is an isomorphism of $k(\mathfrak{p})$, so we may find an element $w^* \in \mathbb{F}_q(x)$ such that ${w^*}^3 = w_0 \mod \mathfrak{p}$ (note also that $v_{\mathfrak{p}}(w^*)=0$ by construction). Again by strong approximation, we choose an element $w_1 \in \mathbb{F}_q(x)$ such that \\
 
  \begin{itemize} \item $v_\mathfrak{p}(w_1 - w^*) \geq 1$, so that $w_1 = w^* \mod \mathfrak{p}$, which implies $w_1^3 = {w^*}^3 =  w_0 \mod \mathfrak{p}$. \item $v_\mathfrak{q}(w_1 - 0) \geq v_\mathfrak{q}(a_1) $ for any finite place $\mathfrak{q}$ such that $v_\mathfrak{q}(a_1) > 0$. \item $v_\mathfrak{q}(w_1) \geq 0$ at all other finite places $\mathfrak{q}$. \end{itemize}
  \vspace{4mm}
  
As in \cite[Lemma 3.18]{MWcubic2}, we let $w_2 = \alpha w_1$. By the same argument as in (Ibid.), we have $v_\mathfrak{p}(a_2) > v_\mathfrak{p}(a_1)$. We now observe how the process behaves at the other finite places of $\mathbb{F}_q(x)$: \\
\begin{itemize} \item If $\mathfrak{q}$ is a finite place such that $v_\mathfrak{q}(a_1) > 0$, then $v_\mathfrak{q}(w_2) =v_\mathfrak{q}(\alpha w_1) = v_\mathfrak{q}(\alpha) + v_\mathfrak{q}(w_1) \geq  v_\mathfrak{q}(a_1).$ Then \begin{align*} v_\mathfrak{q}(ja_1^2 + (w_2^3 + a_1 w_2)) &\geq \min\{v_\mathfrak{q}(ja_1^2 + w_2^3  ) ,v_\mathfrak{q}(a_1 w_2) \} = 2 v_\mathfrak{q}(a_1).\end{align*} Thus $$v_\mathfrak{q}(a_2)  = v_\mathfrak{q}\left(\frac{(ja_1^2 + (w_2^3 + a_1 w_2) )^2}{a_1^3}\right) \geq 4v_\mathfrak{q}(a_1) - v_\mathfrak{q}(a_1^3) = v_\mathfrak{q}(a_1) > 0.$$ \item If $\mathfrak{q}$ is a finite place such that $v_\mathfrak{q}(a_1) = 0$, then $v_\mathfrak{q}(w_1) \geq 0$, so that \begin{align*} v_\mathfrak{q}(ja_1^2 + (w_2^3 + a_1 w_2)) \geq \min\{v_\mathfrak{q}(ja_1^2),v_\mathfrak{q}(w_2^3 + a_1 w_2))\}  = 0. \end{align*} 
 Hence $$v_\mathfrak{q}(a_2)  = v_\mathfrak{q}\left(\frac{(ja_1^2 + (w_2^3 + a_1 w_2) )^2}{a_1^3}\right) \geq 0- v_\mathfrak{q}(a_1^3) = 0.$$ \item If $\mathfrak{q} \neq \mathfrak{p}$ is a finite place of $\mathbb{F}_q(x)$ such that $v_\mathfrak{q}(a_1) < 0$, then $v_\mathfrak{q}(w_2) \geq 0$, so that $$v_\mathfrak{q}(ja_1^2) < v_\mathfrak{q}(w_2^3 + a_1 w_2),$$ whence\begin{align*} v_\mathfrak{q}(ja_1^2 + (w_2^3 + a_1 w_2))  = \min\{v_\mathfrak{q}(ja_1^2),  v_\mathfrak{q}(w_2^3 + a_1 w_2)  = v_\mathfrak{q}(ja_1^2) = 2v_\mathfrak{q}(a_1). \end{align*} So $$v_\mathfrak{q}(a_2)  = v_\mathfrak{q}\left(\frac{(ja_1^2 + (w_2^3 + a_1 w_2) )^2}{a_1^3}\right) =  4v_\mathfrak{q}(a_1)- v_\mathfrak{q}(a_1^3) = v_\mathfrak{q}(a_1).$$   \end{itemize} 
 Summarising, we obtain \\
 
  \begin{itemize} \item $v_\mathfrak{p}(a_2) > v_\mathfrak{p}(a_1)$. \item  If $\mathfrak{q}$ is a finite place such that $v_\mathfrak{q}(a_1) > 0$, then $v_\mathfrak{q}(a_2) > 0$. \item If $\mathfrak{q}$ is a finite place such that $v_\mathfrak{q}(a_1) = 0$, then $v_\mathfrak{q}(a_2) \geq 0$. \item If $\mathfrak{q} \neq \mathfrak{p}$ is a finite place of $\mathbb{F}_q(x)$ such that $v_\mathfrak{q}(a_1) < 0$, then $v_\mathfrak{q}(a_2) = v_\mathfrak{q}(a_1)$. \end{itemize}
  \vspace{4mm}
  
  We now repeat this process as $\mathfrak{p}$ to obtain $y$ and $a$ as desired. Once $a$ is obtained for $\mathfrak{p}$, we may choose another finite place $\mathfrak{p}'$ (if it exists) of $K$ such that $v_{\mathfrak{p}'}(a) < 0$ and $3 \mid  v_{\mathfrak{p}'}(a)$, and perform the same process with initial generator $y$ and equation $y^3 + a y + a^2 = 0$, obtaining some $y'$ and $a'$ such that $$y'^3 + a' y' + a'^2 = 0.$$ In addition to satisfying the requisite valuation conditions at $\mathfrak{p}$, the element $a'$ also satisfies these conditions at $\mathfrak{p}'$, i.e., $v_{\mathfrak{p}'}(a') \geq 0$ or $v_{\mathfrak{p}'}(a') < 0$ and $(v_{\mathfrak{p}'}(a'),3) = 1$. By construction, all other valuations at finite places remain nonnegative or stable in this algorithm. We therefore repeat this process until we find a $b \in \mathbb{F}_q(x)$ and generator $z$ of $L/\mathbb{F}_q(x)$ as desired.\end{proof}
  
  We may now determine $(\partial_{L/\mathbb{F}_q(x)})_{\mathbb{F}_q[x]}$ in terms of the element $b \in \mathbb{F}_q(x)$ of Lemma \ref{3standard}.
  
  \begin{lemme} \label{3disc} Let $b \in \mathbb{F}_q(x)$ be as in Lemma \ref{3standard}. We denote $$b= \frac{\xi_1 \xi_2^2}{\beta},$$ where $\xi_1, \xi_2, \beta \in \mathbb{F}_q[x]$, $\xi_1$ is square-free, and $(\xi_1 \xi_2, \beta) = 1$.  Then \begin{align*} (\partial_{L/\mathbb{F}_q(x)})_{\mathbb{F}_q[x]} = 
\prod_{\mathfrak{p} | (\xi_1)_{\mathbb{F}_q[x]}} \mathfrak{p} \prod_{\mathfrak{p} | (\beta)_{\mathbb{F}_q[x]}} \mathfrak{p}^{-v_\mathfrak{p}(b) + 2} .\end{align*} 

  \end{lemme}
 
\begin{proof} By \cite[Theorem 3.19]{MWcubic2}, whenever $v_\mathfrak{q}(b) < 0$ and $(v_\mathfrak{q}(b) ,3) = 1$, $\mathfrak{q}$ is fully ramified, and by \cite[Theorem 3.20]{MWcubic2}, $e(\mathfrak{Q}|\mathfrak{q}) = 2$ for a single place $\mathfrak{Q}|\mathfrak{p}$ whenever $v_\mathfrak{q}(b) > 0$ and $(v_\mathfrak{q}(b) ,2) = 1$. For the fully ramified places, by \cite[Lemma 3.26]{MWcubic2}, we have $$d ( \mathfrak{P}| \mathfrak{p}) = -v_{\mathfrak{p}}(b )+2.$$
We also have by \cite[Lemma 3.18]{MWcubic2} that $d(\mathfrak{P}|\mathfrak{p}) = 1$ whenever $v_\mathfrak{p}(b) > 0$ and $(v_\mathfrak{p}(b),2) = 1$. These are the two cases where a finite place occurs in the discriminant of $L/\mathbb{F}_q(x)$. In either case, there exists only one place $\mathfrak{P}$ of $L$ above $\mathfrak{p}$ such that $d(\mathfrak{P}|\mathfrak{p})\neq 0$, and for such $\mathfrak{P}$, we again have $f(\mathfrak{P}|\mathfrak{p}) = 1$. Thus we find from \eqref{normdisc} that
 \begin{align*} (\partial_{L/\mathbb{F}_q(x)})_{\mathbb{F}_q[x]} = 
\prod_{\mathfrak{p} | (\xi_1)_{\mathbb{F}_q[x]}} \mathfrak{p} \prod_{\mathfrak{p} | (\beta)_{\mathbb{F}_q[x]}} \mathfrak{p}^{-v_\mathfrak{p}(b) + 2}  \end{align*} \end{proof} 

\section{Explicit triangular integral bases} 
In this section, we demonstrate the existence of an explicit triangular integral basis for a separable cubic extension $L/\mathbb{F}_q(x)$. We divide the argument between $p\neq 3$ and $p=3$, as the construction in each case is different. 

\subsection{$p\neq 3$} 

 Let $y$ be a generator of $L/\mathbb{F}_q(x)$ which satisfies the equation $y^3 - 3y - a=0$, where $a \in \mathbb{F}_q(x)$, which exists by \cite[Corollary 1.3]{MWcubic2}. 
 We let $a= \alpha/(\gamma^3 \beta)$, where $(\alpha,\beta\gamma) = 1$, $\beta$ is cube-free, and $\beta = \beta_1 \beta_2^2$, where $\beta_1$ and $\beta_2$ are square-free. As in \S 1.1, we let $\omega = \gamma \beta_1 \beta_2 y$ and $(4\gamma^6 \beta^2 - \alpha^2)=\eta_1 \eta_2^2$, where $\eta_1$ is square-free. The element $\omega$ is integral over $\mathbb{F}_q[x]$, and the minimal polynomial of $\omega$ over $\mathbb{F}_q(x)$ is equal to
$$X^3 - 3\gamma^2 \beta_1^2 \beta_2^2 X- \alpha \beta_1^2 \beta_2.$$ We let $r$ denote a root of the quadratic resolvent of this cubic polynomial. We also let:
\begin{itemize} 
\item[$\cdot$] If $p\neq 2$, $I= \beta_1^2 \eta_2^2$, 
\item[$\cdot$] If $p= 2$, $I = \beta_1 \alpha_I$, where 
\begin{itemize}
\item $ \alpha_I= \prod_{i=1}^s \alpha_i^{v_{\mathfrak{p}_i}(\alpha) - \frac{1}{2}(\ell_{\mathfrak{p}_i} + 1)}$, when $L/\mathbb{F}_q(x)$ is not Galois,
\item $ \alpha_I= \alpha $, when $L/\mathbb{F}_q(x)$ is Galois,
\end{itemize}
where we write $\alpha$ as a product of irreducible polynomials $\alpha_i$ ($i=1,\ldots,s$) as $$\alpha = \prod_{i=1}^s \alpha_i^{v_{\mathfrak{p}_i}(\alpha) },$$ where $\alpha_i$ corresponds to the (finite) place $\mathfrak{p}_i$ of $\mathbb{F}_q(x)$ and $\ell_{\mathfrak{p}_i}$ is defined as in Lemma \ref{indlemma}.
\end{itemize} 
This machinery allows us to obtain the following description of the explicit triangular integral basis of $\mathcal{O}_{L,x}/\mathbb{F}_q[x]$. 

\begin{theoreme} \label{3primetoqbasis}
Let $p \neq 3$. The set $$\mathfrak{B} = \left\{ 1, \omega+ S, \frac{1}{I} (\omega^2+ T\omega +V  ) \right\}$$ forms an integral basis of $L/ \mathbb{F}_q(x)$, if $S, T,V \in \mathbb{F}_q[x]$, $ V \equiv T^2 -3(\gamma \beta_1 \beta_2)^2 \mod I$, and $T$ is chosen in the following manner, which may be done by the Chinese remainder theorem:
\begin{enumerate} 
\item When $p \neq 2$,  $$T \equiv -\frac{  \alpha  }{2 \gamma^2 \beta_2}  \mod \eta_2^2\quad\text{ and } \quad T\equiv 0 \mod \beta_1^2;$$
\item When $p = 2$, \begin{enumerate} \item If $L/\mathbb{F}_q(x)$ is Galois, $$T \equiv 0 \mod \beta_1 \quad\text{ and }\quad  T \equiv \frac{r}{  \gamma^2 \beta_1^2 \beta_2^2} \mod \alpha_I;$$  \item If $L/\mathbb{F}_q(x)$ is not Galois, $$T \equiv 0 \mod \beta_1 \quad \text{ and }\quad T \equiv \frac{\alpha c}{\gamma^2 \beta_2}\mod \alpha_I.$$ \end{enumerate}
\end{enumerate} 
\end{theoreme} 
\begin{proof}
Since 
\begin{enumerate} 
\item When $p\neq 2$, the result is a direct consequence of \cite[Lemma 6.3]{lotsofpeople} and \S 1.1.
\item Suppose that $p=2$. By \cite[Lemma 3.1, Corollary 3.2]{Scheidler}, which remain valid in characteristic 2, and Lemma \ref{indlemma}, a basis of the form of $\mathfrak{B}$ as in the statement of the theorem
exists if, and only if,  $T, V , S  \in \mathbb{F}_q[x]$ are such that \begin{equation} \label{pneq2T} T^2 +(\gamma \beta_1 \beta_2)^2\equiv 0 \mod  I, \quad \quad T^3 + (\gamma \beta_1 \beta_2)^2T +  \alpha \beta_1^2 \beta_2 \equiv 0 \mod  I^2,\end{equation}
and 
$$ V \equiv T^2 -3(\gamma \beta_1 \beta_2)^2 \mod I.$$
We prove that a choice of $T$ which satisfies the conditions in the statement of the theorem also satisfies \eqref{pneq2T}. 
We note that as $(\beta_1,\alpha_I) = 1$, the first condition in \eqref{pneq2T}, $$T^2 +(\gamma \beta_1 \beta_2)^2\equiv 0 \mod I,$$ 
is equivalent to the two conditions
\begin{equation}\label{I} T^2 \equiv (\gamma \beta_1 \beta_2)^2  \mod \alpha_I \quad \text{ and } \quad T \equiv 0 \mod \beta_1,\end{equation} 
and the second condition in \eqref{pneq2T}  $$T^3 + (\gamma \beta_1 \beta_2)^2T +  \alpha \beta_1^2 \beta_2 \equiv 0 \mod  I^2$$ is equivalent to the two conditions 
\begin{equation} \label{I2} T^3 + (\gamma \beta_1 \beta_2)^2T +  \alpha \beta_1^2 \beta_2 \equiv 0 \mod \alpha_I^2  \quad \text{and} \quad T^3 + (\gamma \beta_1 \beta_2)^2T +  \alpha \beta_1^2 \beta_2 \equiv 0 \mod \beta_1^2. \end{equation} 
First, we note that if $T \equiv 0  \mod \beta_1$, then $T^3 + (\gamma \beta_1 \beta_2)^2T +  \alpha \beta_1^2 \beta_2 \equiv 0 \mod \beta_1^2.$
As $T$ is invertible mod $\alpha_I$ from $(\alpha_I, \gamma \beta_1 \beta_2)=1$, the condition
$$T^3 + (\gamma \beta_1 \beta_2)^2T +  \alpha \beta_1^2 \beta_2 \equiv 0 \mod \alpha_I^2$$ 
is equivalent to 
$$T^4 + (\gamma \beta_1 \beta_2)^2T^2 +  \alpha \beta_1^2 \beta_2 T \equiv 0 \mod \ \alpha_I^2.$$ 
We have 
\begin{align*} 
T^4 + (\gamma \beta_1 \beta_2)^2T^2 +  \alpha \beta_1^2 \beta_2  T \equiv & (\gamma \beta_1 \beta_2)^4 + (\gamma \beta_1 \beta_2)^2T^2 +  \alpha \beta_1^2 \beta_2 T \mod \alpha_I^2\\
\equiv &\beta_1^2 \beta_2 ( \gamma^4 \beta_2^3 \beta_1^2 + \gamma^2 \beta_2 T^2 +  \alpha T)\mod  \alpha_I^2
 \end{align*}
Thus, the condition $$T^3 + (\gamma \beta_1 \beta_2)^2T +  \alpha \beta_1^2 \beta_2 \equiv 0 \mod \ \alpha_I^2$$  
is equivalent to 
$$\gamma^4 \beta_2^3 \beta_1^2 + \gamma^2 \beta_2 T^2 +  \alpha T \equiv 0 \mod  \alpha_I^2.$$ 
 Since $(\gamma^2 \beta_2, \alpha_I)=1$, this condition is in turn equivalent to 
 \begin{equation} \label{T} \gamma^6 \beta_2^4 \beta_1^2 + \gamma^4 \beta_2^2 T^2 +  \alpha \gamma^2 \beta_2 T \equiv 0 \mod  \alpha_I^2\end{equation}
Setting $T'= \gamma^2 \beta_2 T$, we obtain 
$$  T'^2 +  \alpha T' +\gamma^6 \beta_2^4 \beta_1^2 \equiv 0 \mod  \alpha_I^2$$ 
We now note that the quadratic resolvent of the cubic polynomial $X^3 - 3\gamma^2 \beta_1^2 \beta_2^2 X- \alpha \beta_1^2 \beta_2$ is equal to
$$X^2 + \alpha \beta_1^2 \beta_2 X +  \beta_1^4 \beta_2^2 (\gamma^6 \beta_1^2 \beta_2^4 + \alpha^2)$$ 
We denote by $r$ a root of the quadratic resolvent.
Taking $X' = \frac{X}{ \beta_1^2 \beta_2}$, we have that
$$S(X')=X'^2 + \alpha X' + \gamma^6 \beta_1^2 \beta_2^4 + \alpha^2$$ 
is an integral polynomial over $\mathbb{F}_q[x]$. We now divide the argument into the cases where $L/\mathbb{F}_q(x)$ is Galois or not.
\begin{enumerate}[(a)]
\item If $L/\mathbb{F}_q(x)$ is Galois, then $r \in \mathbb{F}_q(x)$, setting $r'=\frac{r}{ \beta_1^2 \beta_2}$,
$S(r')=0$
and 
$$r'^2 +  \alpha r' +\gamma^6 \beta_2^4 \beta_1^2 =-\alpha^2 $$
Choosing $T\equiv \frac{r}{\gamma^2  \beta_1^2 \beta_2^2} \mod \alpha_I$ and $T\equiv 0 \mod \beta_1$ via the Chinese remainder theorem, $T$ satisfies \eqref{pneq2T}, proving the theorem.
\item If $L/\mathbb{F}_q(x)$ is not Galois, then $r \notin \mathbb{F}_q(x)$, and the base change $X''= \frac{X'}{\alpha}$ transforms $S(X')$ into the Artin-Schreier polynomial 
$$X''^2 + X''+ \frac{\gamma^6 \beta_1^2 \beta_2^4 }{\alpha^2} +1 $$ 
Via \cite[Example 5.8.9]{Vil}, we let $c$ be chosen such that $$v_{\mathfrak{p}_i}\left(c^2 +c +\frac{\gamma^6 \beta_1^2 \beta_2^4 }{\alpha^2} +1 \right) =- \ell _{\mathfrak{p}_i}$$ if $\mathfrak{p}_i\mid (\alpha)_{\mathbb{F}_q[x]}$ is ramified in $\mathbb{F}_q(x)(r)$ and  $v_{\mathfrak{p}} (c^2 +c +\frac{\gamma^6 \beta_1^2 \beta_2^4 }{\alpha^2} +1 )  \geq 0$ for all other finite places $\mathfrak{p}$ of $\mathbb{F}_q(x)$. 
It follows that $c' = \alpha c $ is such that 
\begin{equation} \label{firststandard} v_{\mathfrak{p}_i}({c'}^2 +\alpha c' +\gamma^6 \beta_1^2 \beta_2^4  +\alpha^2) = - \ell_{\mathfrak{p}_i} + 2 v_{\mathfrak{p}_i} ( \alpha)\end{equation} if $\mathfrak{p}_i\mid (\alpha)_{\mathbb{F}_q[x]}$ is ramified in $\mathbb{F}_q(x)(r)$, and \begin{equation} \label{secondstandard} v_{\mathfrak{p}}({c'}^2 +\alpha c' +\gamma^6 \beta_1^2 \beta_2^4  +\alpha^2) \geq 2 v_{\mathfrak{p}} ( \alpha)\end{equation} for all other finite places $\mathfrak{p}$ of $\mathbb{F}_q(x)$. As $\ell_{\mathfrak{p}_i} > 0$ if $\mathfrak{p}_i\mid (\alpha)_{\mathbb{F}_q[x]}$ is ramified in $\mathbb{F}_q(x)(r)$, it follows for such $\mathfrak{p}_i$ by the non-Archimedean triangle inequality and \eqref{firststandard} that $$v_{\mathfrak{p}_i}({c'}^2 +\alpha c' +\gamma^6 \beta_1^2 \beta_2^4  ) = \min \{- \ell_{\mathfrak{p}_i} + 2 v_{\mathfrak{p}_i} ( \alpha),2 v_{\mathfrak{p}_i} ( \alpha) \}  = - \ell_{\mathfrak{p}_i} + 2 v_{\mathfrak{p}_i} ( \alpha) \geq -(\ell_{\mathfrak{p}_i} +1 ) +  2 v_{\mathfrak{p}_i} ( \alpha).$$ 
For all other places $\mathfrak{p}$ of $\mathbb{F}_q(x)$ dividing $(\alpha)_{\mathbb{F}_q[x]}$, we have $\ell_{\mathfrak{p}} = -1$. By the non-Archimedean triangle inequality and \eqref{secondstandard}, we obtain $$v_{\mathfrak{p}}({c'}^2 +\alpha c' +\gamma^6 \beta_1^2 \beta_2^4  ) \geq \min\{2 v_{\mathfrak{p}}( \alpha),2 v_{\mathfrak{p}}( \alpha)\} = 2 v_{\mathfrak{p}}( \alpha) = -(\ell_{\mathfrak{p}} +1 ) +  2 v_{\mathfrak{p}} ( \alpha).$$
Thus, for this this choice of $c'$, we have that 
$$  {c'}^2 +  \alpha c' +\gamma^6 \beta_2^4 \beta_1^2 \equiv 0 \mod  \alpha_I^2.$$ 
We now choose $d = \frac{c'}{\gamma^2 \beta_2}$, which yields
$$\gamma^6 \beta_2^4 \beta_1^2 + \gamma^4 \beta_2^2 d^2 +  \alpha \gamma^2 \beta_2 d \equiv 0 \mod  \alpha_I^2$$ 
So that $d$ is a solution of \ref{T}. Finally, via the Chinese remainder theorem, we choose
$$T \equiv 0 \mod \beta_1 \text{ and } T \equiv d \mod \alpha_I,$$ so that \eqref{pneq2T} is again satisfied, proving the theorem. 
\end{enumerate}
\end{enumerate}
\end{proof}

\section{$p=3$} 

We now consider the construction of an explicit triangular integral basis when $p=3$. We let $z$ be a generator of $L/\mathbb{F}_q(x)$ with minimal polynomial $$X^3 + b X + b^2 = 0,$$ where $b$ is as in Lemma \ref{3standard}. We write $$b= \frac{\xi_1 \xi_2^2}{\beta} \quad \text{ and } \quad \beta = \prod_{i=1}^s p_i^{\ell_{\mathfrak{p}_i}}$$ where $p_i\in \mathbb{F}_q[x]$ distinct irreducible polynomials ($i=1,\ldots,s$), with $p_i$ corresponding to the (finite) place $\mathfrak{p}_i$ of $\mathbb{F}_q(x)$, where also $\xi_1, \xi_2  \in \mathbb{F}_q[x]$, $\xi_1$ is square-free, and $(\xi_1 \xi_2, \beta ) = 1$, and $\ell_{\mathfrak{p}_i}>0$, $(\ell_{\mathfrak{p}_i}, 3 ) =1$ for each $i=1,\ldots,s$.

\begin{theoreme} \label{3dividesqbasis}
Let $p=3$. For $j=1,2$, let $$P_j = \prod_{i=1}^s p_i^{1+ \left\lfloor \frac{j 2\ell_{\mathfrak{p}_i}}{3}\right\rfloor},$$ where $\left\lfloor \frac{j 2\ell_{\mathfrak{p}_i}}{3}\right\rfloor$ is the integral part of  $\frac{j2 \ell_{\mathfrak{p}_i}}{3}$. Then the set $$\mathfrak{B}=\left\{ \frac{P_2}{\xi_1\xi_2^2} z^2 , \frac{P_1}{\xi_2} z, 1 \right\}$$ is an integral basis for $L/\mathbb{F}_q(x)$. 
\end{theoreme}
\begin{proof} 
Let $\mathfrak{y} = \frac{z}{\xi_2}$ and $\gamma =\frac{\xi_1}{\beta}$. We have 
$$\mathfrak{y}^3 + \gamma \mathfrak{y} +  \gamma^2 \xi_2=0.$$
The discriminant of $\mathfrak{y}$ is equal to $\Delta (\mathfrak{y})=2 \gamma^3$, and by Lemma \ref{3disc}, we have \begin{align*} (\partial_{L/\mathbb{F}_q(x)})_{\mathbb{F}_q[x]} = 
\prod_{\mathfrak{p} | (\xi_1)_{\mathbb{F}_q[x]}} \mathfrak{p} \prod_{\mathfrak{p} | (\beta)_{\mathbb{F}_q[x]}} \mathfrak{p}^{-v_\mathfrak{p}(b) + 2}  = 
(\xi_1)_{\mathbb{F}_q[x]} \prod_{\mathfrak{p} | (\beta)_{\mathbb{F}_q[x]}} \mathfrak{p}^{-v_\mathfrak{p}(b) + 2}= (\xi_1)_{\mathbb{F}_q[x]} \prod_{\mathfrak{p} | (\beta)_{\mathbb{F}_q[x]}} \mathfrak{p}^{-\ell_{\mathfrak{p}} + 2} \end{align*}


We may thus write $$\mathfrak{B}=\left\{ \frac{P_2}{\xi_1} \mathfrak{y}^2 , P_1 \mathfrak{y}, 1 \right\}.$$ We let $\beta = \beta_1 \beta_2^2$, where $\beta_1$ is square-free. Let $g$ be chosen in the algebraic closure of $\mathbb{F}_q(x)$ to satisfy $g^2 =- \gamma$. By Kummer theory (see \cite[Example 5.8.9]{Vil}), the finite places $\mathfrak{p}$ of $ \mathbb{F}_q(x)$ ramified in $  \mathbb{F}_q(x)(g)$ are those such that $\mathfrak{p} | ( \xi_1 \beta_1)_{\mathbb{F}_q[x]}$. 
In what follows, $\mathfrak{p}$ denotes a finite place of $\mathbb{F}_q(x)$, $\mathfrak{p}_g$ a place of $\mathbb{F}_q(x)(g)$ above $\mathfrak{p}$, $\mathfrak{P}$ a place of $L$ above $\mathfrak{p}$ and $\mathfrak{P}_g$ a place of $L(g)$ above $\mathfrak{p}_g$ and $\mathfrak{P}$.

By definition, the element $ \frac{\mathfrak{y}}{g}$ satisfies the following Artin-Schreier equation above $\mathbb{F}_q(x)(g)$:
$$\left( \frac{\mathfrak{y}}{g}\right)^3 - \frac{\mathfrak{y}}{g} +  g \xi_2=0.$$ 
Let $\mathfrak{p}| (\beta_1 )_{\mathbb{F}_q[x]} $. Thus, $v_{\mathfrak{p}_g} ( g)= v_{\mathfrak{p}} ( \gamma )<0$ and $(v_{\mathfrak{p}_g} ( g),3)=1$. Hence $e(\mathfrak{P}_g | \mathfrak{p}_g)=3$ by Artin-Schreier theory, and by the non-Archimedean triangular inequality, $v_{\mathfrak{P}_g} (\frac{\mathfrak{y}}{g} ) = v_{\mathfrak{p}} ( \gamma )$, so that
\begin{align*} v_{\mathfrak{P}_g} (\mathfrak{y} ) &= v_{\mathfrak{p}} ( \gamma )+ v_{\mathfrak{P}_g} ( g)\\
&= v_{\mathfrak{p}} ( \gamma )+ e(\mathfrak{P}_g| \mathfrak{p}_g) v_{\mathfrak{p}_g} ( g)\\
&= v_{\mathfrak{p}} ( \gamma )+ 3v_{\mathfrak{p}} ( \gamma )\\
 &= 4v_{\mathfrak{p}} ( \gamma )\end{align*}
and 
$$v_{\mathfrak{P}} (\mathfrak{y} )=2 v_{\mathfrak{p}} ( \gamma )=-2\ell_\mathfrak{p}.$$
Then, for $j=1,2$, 
$$v_{\mathfrak{P}} \left(P_j \frac{\mathfrak{y}^j}{\xi_1}\right) = jv_{\mathfrak{P}} (\mathfrak{y} ) + v_{\mathfrak{P}} (P_j )= -2j \ell_{\mathfrak{p}}+3 \left(1+ \left\lfloor \frac{j 2\ell_{\mathfrak{p}}}{3}\right\rfloor \right) \geq 0$$
Suppose that $\mathfrak{p}| (\beta)_{\mathbb{F}_q[x]} $ and $\mathfrak{p}\nmid (\beta_1)_{\mathbb{F}_q[x]} $. Then $v_{\mathfrak{p}_g} ( g)= \frac{v_{\mathfrak{p}} ( \gamma )}{2}<0$ and $(v_{\mathfrak{p}_g} ( g),3)=1$. Hence again $e(\mathfrak{P}_g | \mathfrak{p}_g)=3$ by Artin-Schreier theory, and by the non-Archimedean triangle inequality, $v_{\mathfrak{P}_g} (\frac{\mathfrak{y}}{g} ) = \frac{v_{\mathfrak{p}} ( \gamma )}{2}$, so that
\begin{align*} v_{\mathfrak{P}_g} (\mathfrak{y} ) &= \frac{v_{\mathfrak{p}} ( \gamma )}{2}+ v_{\mathfrak{P}_g} ( g)\\ &= \frac{v_{\mathfrak{p}} ( \gamma )}{2}+ 3 \frac{v_{\mathfrak{p}} ( \gamma)}{2}\\ &= 2v_{\mathfrak{p}} ( \gamma )\end{align*}
and 
$$v_{\mathfrak{P}} (\mathfrak{y} )=2 v_{\mathfrak{p}} ( \gamma )=-2\ell_\mathfrak{p}.$$
Then, for $j=1,2$, $$v_{\mathfrak{P}} \left(P_j \frac{\mathfrak{y}^j}{\xi_1}\right) = 2v_{\mathfrak{P}} (\mathfrak{y} ) + v_{\mathfrak{P}} (P_j )= -2j \ell_{\mathfrak{p}}+3 \left(1+ \left\lfloor \frac{j 2\ell_{\mathfrak{p}}}{3}\right\rfloor \right)\geq 0.$$
If $\mathfrak{p}| (\xi_1)_{\mathbb{F}_q[x]}$, then $v_{\mathfrak{p}_g} ( g)= v_{\mathfrak{p}} ( \gamma )= v_{\mathfrak{p}} ( \xi_1)=1$, as $\xi_1$ is square free. By the non-Archimedean triangle inequality, we find that either $v_{\mathfrak{P}_g} ( \frac{\mathfrak{y}}{g}) =0$ or $v_{\mathfrak{P}_g} ( \frac{\mathfrak{y}}{g}) =v_{\mathfrak{p}} ( \gamma \xi_2 )$. In both cases, $v_{\mathfrak{P}_g} ( \mathfrak{y}) \geq v_{\mathfrak{P}_g} ( g)>0$, whence $v_{\mathfrak{P}} ( \mathfrak{y}) >0$ and $v_{\mathfrak{P}} ( \mathfrak{y}) \geq 1 = v_{\mathfrak{p}} ( \xi_1)$. Thus, we obtain
 $$v_{\mathfrak{P}} \left(P_2 \frac{\mathfrak{y}^2}{\xi_1}\right) = 2v_{\mathfrak{P}} (\mathfrak{y} )- v_{\mathfrak{P}} (\xi_1 ) =2v_{\mathfrak{P}} (\mathfrak{y} )- e(\mathfrak{P}| \mathfrak{p}) v_{\mathfrak{p}} (\xi_1 )\geq 0,$$ 
where $e(\mathfrak{P}| \mathfrak{p})=2$ by \cite[Theorem 3.20]{MWcubic2}.

Finally, we obtain $$\frac{P_2 P_1 }{\xi_1}= \frac{\prod_{i=1}^s p_i^{2+\sum_{j=1}^2 \left\lfloor \frac{j 2\ell_{\mathfrak{p}_i}}{3}\right\rfloor}}{\xi_1}.$$ We have
 $$2+\sum_{j=1}^2 \left\lfloor \frac{j 2\ell_{\mathfrak{p}_i}}{3}\right\rfloor=2+ 2\ell_{\mathfrak{p}_i} -1 =2\ell_{\mathfrak{p}_i} +1, $$ 
 whence $$\frac{P_2 P_1 }{\xi_1}= \frac{\prod_{i=1}^s p_i^{2\ell_{\mathfrak{p}_i} +1 }}{\xi_1}.$$
We therefore find via the previous argument and Lemma \ref{3disc} that
\begin{align*}
\left(\left( \frac{P_2 P_1 }{\xi_1} \right)^2  \frac{\xi_1^3 }{\beta^3}\right)_{\mathbb{F}_q[x]}&=\left(\frac{\xi_1\prod_{i=1}^s p_i^{4\ell_{\mathfrak{p}_i} +2 } }{\beta^3}\right)_{\mathbb{F}_q[x]} =\left(\xi_1\prod_{i=1}^s p_i^{\ell_{\mathfrak{p}_i} +2 }\right)_{\mathbb{F}_q[x]}\\
&=\left(\xi_1\prod_{i=1}^s p_i^{\ell_{\mathfrak{p}_i} +2}\right)_{\mathbb{F}_q[x]} = (\partial_{L/\mathbb{F}_q(x)})_{\mathbb{F}_q[x]},
\end{align*} which proves that $\mathfrak{B}$ is an integral basis.

\end{proof}

 \bibliographystyle{plain}
\raggedright
\bibliography{references}
\vspace{.5cm}

\end{document}